\documentclass[a4paper]{amsart}                % American Mathematical Society document class for articles
\usepackage[latin1]{inputenc}                  % Codificación europea del teclado (acentos, símbolos teclado ...)
\usepackage[T1]{fontenc}                       % Acentos de palabras y división silábica
\usepackage[spanish,english]{babel}            % Idioma principal el ingles
\usepackage{graphicx}                  % Incorporar gráficos
\usepackage{amsmath,amssymb,amsthm}            % Paquetes de la American Mathematical Society
\usepackage{latexsym}                          % Símbolos
\usepackage{delarray}                          % Delimiter array: permite integrar delimitadores en las opciones del entorno array
\usepackage{bbm}                               % Fuentes para los números reales, racionales, ...
\usepackage{hyperref}                          % Enlaces de "hipertexto"
\usepackage{calrsfs,eufrak,mathrsfs,upgreek}   % Tipo de letra
\usepackage{bbding,trfsigns}                   % Símbolos
\usepackage{wasysym}                           % Variar el tamaño de los símbolos
\usepackage{datetime}                          % Imprimir fecha y hora

\DeclareMathAlphabet{\mathpzc}{OT1}{pzc}{m}{it} % Tipo de letra \mathpzc

\newtheorem{teo}{Theorem}

\newtheorem{propo}{Proposition}

\author[J. Betancor]{J.J. Betancor}
\address{Departamento de An\'{a}lisis Matem\'{a}tico\\
Universidad de la Laguna\\
Campus de Anchieta, Avda. Astrof\'{\i}sico Francisco S\'{a}nchez, s/n\\
38271 La Laguna (Sta. Cruz de Tenerife), Espa\~na}
\email{jbetanco@ull.es, jcfarina@ull.es, lrguez@ull.es}

\author[R. Crescimbeni]{R. Crescimbeni}
\address{Departamento de Matem\'aticas, Facultad de Econom\'{\i}a y Administraci\'on, Universidad Nacional de Co\-ma\-hue, 8300 Neuqu\'en, Argentina}
\email{rcrescim@uncoma.edu.ar}

\author[J.C. Fari\~{n}a]{J.C. Fari\~{n}a}
\author[L. Rodr\'{\i}guez-Mesa]{L. Rodr\'{\i}guez-Mesa}
\thanks{The first, third and fourth author are partially supported by MTM2010-17974. The second author is partially supported by the Universidad Nacional de Comahue.}
\date{\today}

\begin{document}

\title[]
{Multipliers and imaginary powers of the Schr\"odinger operators characterizing UMD Banach spaces}

 \keywords{}
\begin{abstract}
In this paper we establish $L^p$-boundedness properties for Laplace type transform spectral multipliers associated  with the Schr\"odinger operator $\mathcal{L}=-\Delta +V$. We obtain for this type of multipliers pointwise representation as principal value integral operators. We also characterize the UMD Banach spaces in terms of the $L^p$-boundedness of the imaginary powers $\mathcal{L}^{i\gamma }$, $\gamma \in \mathbb{R}$, of $\mathcal{L}$.

\end{abstract}

\maketitle

\section{Introduction}
We study certain class of spectral multipliers, usually called Laplace transform type multipliers associated with the Schr\"odinger operator $\mathcal{L}=-\Delta +V$, where $\Delta $ represents the Laplacian operator and the potential $V\geq 0$ satisfies reverse H\"older inequalities. We prove that $L^p$-boundedness of some of those multipliers, the imaginary powers of $\mathcal{L}$, acting on Banach valued functions characterizes the UMD property for the Banach space.

We now recall some definitions and properties that will be useful in order to state and to prove our results.

We consider the Schr\"odinger operator $\mathcal{L}=-\Delta +V$ on $\mathbb{R}^n$, with $n\geq 3$. We assume that $V \geq 0$ is a locally integrable function on $\mathbb{R}^n$ belonging to the class $B_q$, that is, there exists $C>0$ such that, for every ball $B$ in $\mathbb{R}^n$,
$$
\left(\frac{1}{|B|}\int_BV^q(x)dx\right)^{1/q}\leq C\frac{1}{|B|}\int_BV(x)dx,
$$
for some $q\geq n/2$.

The operator $\mathcal{L}$, suitably understood, is a closed unbounded and positive operator in $L^2(\mathbb{R}^n)$. Then, there exists the spectral measure $E_\mathcal{L}$ associated with $\mathcal{L}$ and, for every measurable bounded function $m$ on $[0,\infty )$, we define the spectral multiplier $T_m^\mathcal{L}$ by
$$
T_m^\mathcal{L}(f)=\int_{[0,\infty )}m(\lambda )E_\mathcal{L}(d\lambda )f,\quad f\in L^2(\mathbb{R}^n).
$$
It is well known that $T_m^\mathcal{L}$ defines a bounded operator from $L^2(\mathbb{R}^n)$ into itself.

We say that a measurable function $m$ on $(0,\infty )$ is of Laplace transform type when $m(\lambda )=\lambda \int_0^\infty e^{-\lambda t}\phi (t)dt$, $\lambda \in (0,\infty )$, for a certain $\phi \in L^\infty (0,\infty )$. The spectral multiplier $T_m^\mathcal{L}$ is called of Laplace transform type when the function $m$ is of Laplace transform type.

For every $t>0$, we define the operator $W_t^\mathcal{L}$ by
$$
W_t^\mathcal{L}(f)=\int_{[0,\infty )}e^{-\lambda t}E_\mathcal{L}(d\lambda )f,\quad f\in L^2(\mathbb{R}^n).
$$

The uniparametric family $\{W_t^\mathcal{L}\}_{t>0}$ is the semigroup of operators generated by $-\mathcal{L}$ in $L^2(\mathbb{R}^n)$. For every $t>0$ and $1\leq p<\infty$ the operator $W_t^\mathcal{L}(f)$ can be extended from $L^p(\mathbb{R}^n)\cap L^2(\mathbb{R}^n)$ to $L^p(\mathbb{R}^n)$ as a contraction from $L^p(\mathbb{R}^n)$ into itself.  The semigroup of operators
$\{W_t^\mathcal{L}\}_{t>0}$ is not conservative. Moreover, we can write, for every $t>0$ and $f\in L^p(\mathbb{R}^n)$, $1\leq p<\infty$,
$$
W_t^\mathcal{L}(f)(x)=\int_{\mathbb{R}^n}W_t^\mathcal{L}(x,y)f(y)dy,\quad x\in \mathbb{R}^n,
$$
where $W_t^\mathcal{L}(x,y)$, $x,y\in \mathbb{R}^n$ and $t\in (0,\infty )$, is a $C^\infty (\mathbb{R}^n\times \mathbb{R}^n\times (0,\infty ))$ such that, according to Feynman-Kac property,
$$
|W_t^\mathcal{L}(x,y)|\leq CW_t(x,y),\quad x,y\in \mathbb{R}^n\mbox{ and }t>0,
$$
being
$$
W_t(x,y)=\frac{1}{(4\pi t)^{n/2}}e^{-\frac{|x-y|^2}{4t}},\quad x,y\in \mathbb{R}^n\mbox{ and }t>0.
$$
We establish a pointwise representation of the Laplace transform type operator $T_m^\mathcal{L}$ as a principal value integral operator and we prove $L^p$-boundedness properties of $T_m^\mathcal{L}$. As usual by $C_c^ \infty (\mathbb{R}^n)$ we denote the space of smooth functions with compact support in $\mathbb{R}^n$.

\begin{teo}\label{Valorprincipal}
Suppose that $m(\lambda )=\lambda \int_0^\infty e^{-\lambda t}\phi (t)dt$, $\lambda \in (0,\infty )$, where $\phi \in L^\infty (0,\infty )$. Then, for every $f\in C_c^\infty (\mathbb{R}^n)$,
\begin{equation}\label{VP1}
T_m^\mathcal{L}(f)(x)=\lim_{\varepsilon \rightarrow 0^+}\left(\alpha (\varepsilon )f(x)+\int_{|x-y|>\varepsilon }K_\phi ^\mathcal{L}(x,y)f(y)dy\right),\quad \mbox{a.e. }x\in\mathbb{R}^n,
\end{equation}
where
$$
K_\phi ^\mathcal{L}(x,y)=-\int_0^\infty \phi (t)\frac{\partial}{\partial t}W_t^\mathcal{L}(x,y)dt,\quad x,y\in \mathbb{R}^n, x\not=y,
$$
and $\alpha$ is a certain measurable bounded function on $(0,\infty )$. Moreover, if there exists the limit $\lim_{t\rightarrow 0^+}\phi (t)=\phi (0^+)$, then
\begin{equation}\label{VP2}
T_m^\mathcal{L}(f)(x)=\phi (0^+)f(x) +\lim_{\varepsilon \rightarrow 0^+}\int_{|x-y|>\varepsilon }K_\phi ^\mathcal{L}(x,y)f(y)dy,\quad \mbox{a.e. } x\in \mathbb{R}^n.
\end{equation}
\end{teo}

\begin{teo}\label{Lp}
Suppose that $m(\lambda )=\lambda \int_0^\infty e^{-\lambda t}\phi
(t)dt$, $\lambda \in (0,\infty )$, where $\phi \in L^\infty
(0,\infty )$. Then $T_m^\mathcal{L}$ can be extended to $L^p(
\mathbb{R}^n)$ as a bounded operator from $L^p(\mathbb{R}^n)$ into
itself, for every $1<p<\infty $, and as a bounded operator from
$L^1(\mathbb{R}^n)$ into $L^{1,\infty }(\mathbb{R}^n)$. Moreover,
this extension can be given by (\ref{VP1}) and, when  the limit
$\lim_{t \rightarrow 0^+}\phi(t)=\phi(0^+)$ exists, by (\ref{VP2}),
for every $f \in L^p(\mathbb R^n)$, $1 \leq p< \infty$.
\end{teo}

Note that since the semigroup of operators $\{W_t^\mathcal{L}\}_{t>0}$ is not conservative, the $L^p$-boundedness, $1<p<\infty$, of the Laplace transform type multipliers $T_m^\mathcal{L}$ can not be deduced from the result established in \cite[p. 121]{Stein}. The harmonic analysis operators (maximal operators, Riesz transforms  and Littlewood-Paley g-functions) in the Schr\"odinger setting have been studied in $L^p$-spaces by several authors in last years (see for instance, \cite{AST}, \cite{DGMTZ} and \cite{Sh1}). In order to show Theorems \ref{Valorprincipal} and \ref{Lp}, inspired in the procedure developed by Shen \cite{Sh1} to analyze Riesz transforms, we take advantage that $\mathcal{L}$ is a ``nice" perturbation of the Laplacian operator $-\Delta$. This fact allows us to write the multipliers in the Schr\"odinger setting, in some local sense, as perturbation of the corresponding multipliers associated to the Laplacian.

In the localization of our operators the function $\rho$ defined in \cite[p. 516]{Sh1} by
$$
\rho (x)=\sup\{r>0: \frac{1}{r^{n-2}}\int_{B(x,r)}V(y)dy\leq 1\},\,\,\,x\in \mathbb{R}^n,
$$
plays an important role. The main properties of this function $\rho$ can be encountered in \cite[\S 1]{Sh1}. We also use several properties of the heat kernel $W_t^\mathcal{L}(x,y)$ associated to the Schr\"odinger operator $\mathcal{L}$ that can be found, for instance, in \cite{DGMTZ}.

A Banach space $B$ is said to be  UMD when the Hilbert transform $H$ defined in a natural way on $L^p(\mathbb{R}^n)\bigotimes B$ can be extended to $L^p(\mathbb{R}^n,B)$ as a bounded operator from
$L^p(\mathbb{R}^n,B)$ into itself for some (equivalently, for every) $1<p<\infty$ (see \cite{Bo1} and \cite{Bu1}). Here, for every $1\leq p<\infty$, by $L^p(\mathbb{R}^n, B)$ we represent the Bochner-Lebesgue space of exponent $p$. UMD property is related to geometric properties of Banach spaces (\cite{Bu3}). In last years several authors have established connections between geometry of Banach spaces and harmonic analysis. In particular, characterizations of UMD, convexity or smoothness properties of a Banach space have been given in terms of $L^p$-boundedness of certain singular integrals or Littlewood-Paley g-functions (\cite{AST}, \cite{Hy1}, \cite{Hy2}, \cite{MTX} and \cite{Xu}). Here, inspired in the results of Guerre-Delabriere \cite{GD} related to the imaginary powers of the Laplacian, we characterize the Banach spaces having the UMD property by the $L^p$-boundedness of the imaginary power $\mathcal{L}^{i\gamma }$, $\gamma \in \mathbb{R}$, of the Schr\"odinger operator.

Let $\gamma \in \mathbb{R}$. We denote by $m_\gamma$ the function $m_\gamma (\lambda )=\lambda ^{i\gamma }$, $\lambda \in (0,\infty )$. It is clear that $m_\gamma (\lambda )=\lambda \int_0^\infty e^{-\lambda t}\phi _\gamma (t)dt$, $\lambda \in (0,\infty )$, where $\phi _\gamma (t)=t^{-i\gamma }/\Gamma (1-i\gamma )$, $t\in (0,\infty )$. We define, as usual, the imaginary power $\mathcal{L}^{i\gamma }$ of $\mathcal{L}$ by
$$
\mathcal{L}^{i\gamma }=T_{m_\gamma}^\mathcal{L}.
$$

According to Theorem \ref{Lp},  $\mathcal{L}^{i\gamma}$ can be extended to $L^p(\mathbb{R}^n)$ as a bounded operator from $L^p(\mathbb{R}^n)$ into itself, for every $1<p<\infty$. If $B$ is a Banach space we define $\mathcal{L}^{i\gamma}$ on $L^p(\mathbb{R}^n)\bigotimes B$, $1<p<\infty$, in a natural way.

\begin{teo}\label{UMD}
Let $B$ be a Banach space. Then, the following properties are equivalent:

(i) $B$ is a UMD space.

(ii) For every $\gamma \in \mathbb{R}$ and for some (equivalently, for any) $1<p<\infty$, the operator $\mathcal{L}^{i\gamma }$ can be extended to $L^p(\mathbb{R}^n,B)$ as a bounded operator from $L^p(\mathbb{R}^n,B)$ into itself.
\end{teo}

This paper is organized as follows. In Section 2 we present a proof of Theorem \ref{Valorprincipal}. Theorems \ref{Lp} and \ref{UMD} are proved in Section 3. Finally, we present in the Appendix, for the sake of completeness, a proof of a version of Theorem \ref{Valorprincipal} in the Laplacian (classical) case.

Throughout this paper by $C$ and $c$ we always denote positive constants that can change in each occurrence.

\section{Proof of Theorem \ref{Valorprincipal}}

Assume that $\phi \in L^\infty (0,\infty )$ and define the function $m$ as follows
$$
m(\lambda )=\lambda \int_0^\infty e^{-\lambda v}\phi (v)dv,\quad \lambda \in [0,\infty ).
$$

It is clear that $m$ is also in $L^\infty (0,\infty )$. The spectral multiplier $T_m^\mathcal{L}$ in the Schr\"odinger setting associated with $m$ is defined by
$$
T_m^\mathcal{L}(f)=\int_{[0,\infty )}m(\lambda )E_\mathcal{L}(d\lambda )f,\quad f\in L^2(\mathbb{R}^n),
$$
where $E_\mathcal{L}$ represents the spectral measure for the Schr\"odinger operator $\mathcal{L}$. It is well known that $T_m^\mathcal{L}$ is a bounded operator from $L^2(\mathbb{R}^n)$ into itself.

We are going to prove Theorem \ref{Valorprincipal}.
Let $f,g\in C_c^\infty(\mathbb{R}^n)$. We can write
$$
\langle T_m^\mathcal{L}(f),g\rangle=\left\langle\int_{[0,\infty )}m(\lambda )E_\mathcal{L}(d\lambda )f,g\right\rangle=\int_{[0,\infty )}m(\lambda )d\mu _{f,g;\mathcal{L}}(\lambda ),
$$
where by $\mu _{f,g;\mathcal{L}}$ we denote the measure defined by
$$
\mu _{f,g;\mathcal{L}}(A)=\langle E_\mathcal{L}(A)f,g\rangle,
$$
for every Borel set $A\subset [0,\infty )$. The set function $\mu _{f,g;\mathcal{L}}$ is a complex measure on $[0,\infty )$ satisfying that $|\mu _{f,g;\mathcal{L}}|([0,\infty ))\leq ||f||_2||g||_2$, where $|\mu _{f,g;\mathcal{L}}|$ represents the total variation measure of $\mu _{f,g;\mathcal{L}}$.

We have that
\begin{eqnarray*}
\langle T_m^\mathcal{L}(f),g\rangle&=&\int_{[0,\infty )}\lambda \int_0^\infty e^{-\lambda v}\phi (v)dvd\mu _{f,g;\mathcal{L}}(\lambda )\\
&=&\int_0^\infty \phi (v)\int_{[0,\infty )}\lambda e^{-\lambda v}d\mu _{f,g;\mathcal{L}}(\lambda )dv\\
&=&\int_0^\infty \phi (v)\int_{[0,\infty )}\left(-\frac{\partial}{\partial v}\right)(e^{-\lambda v})d\mu _{f,g;\mathcal{L}}(\lambda )dv.
\end{eqnarray*}

Here, we can interchange the order of integration because
$$
\int_{[0,\infty )}\int_0^\infty \lambda e^{-\lambda v}|\phi (v)|dvd|\mu _{f,g;\mathcal{L}}|(\lambda )\leq ||\phi ||_\infty |\mu _{f,g;\mathcal{L}}|([0,\infty ))<\infty .
$$

Since
$$
\left|\frac{e^{-\lambda (v+h)}-e^{-\lambda v}}{h}\right|\leq \lambda e^{\lambda (|h|-v)}\leq \lambda e^{-\lambda v/2},\quad v,\lambda >0 \mbox{ and }|h|<\frac{v}{2},
$$
and
$$
\int_{[0,\infty )}\lambda e^{-\lambda v/2}d|\mu _{f,g;\mathcal{L}}|(\lambda )\leq \frac{2}{v}|\mu _{f,g;\mathcal{L}}|([0,\infty ))<\infty ,\quad v>0,
$$
we can differentiate under the integral sign and write
\begin{eqnarray*}
\langle T_m^\mathcal{L}(f),g\rangle &=&\int_0^\infty \phi (v)\left(-\frac{d}{dv}\right)\int_{[0,\infty )}e^{-\lambda v}d\mu _{f,g;\mathcal{L}}(\lambda )dv\\
&=&\int_0^\infty \phi (v)\left(-\frac{d}{dv}\right)<W_v^\mathcal{L}(f),g>dv\\
&=&\int_0^\infty \phi (v)\left(-\frac{d}{dv}\right)\int_{\mathbb{R}^n}\int_{\mathbb{R}^n}W_v^\mathcal{L}(x,y)f(y)dy\overline{g(x)}dxdv.
\end{eqnarray*}
We have that
$$
\int_{\mathbb{R}^n}\int_{\mathbb{R}^n}|W_v^\mathcal{L}(x,y)-W_v(x,y)||f(y)||g(x)|dydx<\infty, \quad v\in (0,\infty ),
$$
and
$$
\int_{\mathbb{R}^n}\int_{\mathbb{R}^n}\left|\frac{\partial}{\partial v}(W_v^\mathcal{L}(x,y)-W_v(x,y))\right||f(y)||g(x)|dydx<\infty, \quad v\in (0,\infty ).
$$

Indeed, according to \cite[(2.2) and (2.7)]{DGMTZ} it follows that
\begin{eqnarray*}
\lefteqn{\int_{\mathbb{R}^n}\int_{\mathbb{R}^n}|W_v^\mathcal{L}(x,y)-W_v(x,y)||f(y)||g(x)|dydx }\\
&& +\int_{\mathbb{R}^n}\int_{\mathbb{R}^n}\left|\frac{\partial}{\partial v}(W_v^\mathcal{L}(x,y)-W_v(x,y))\right||f(y)||g(x)|dydx\\
&\leq C&\frac{1+v}{v^{n/2+1}}\int_{\mathbb{R}^n}\int_{\mathbb{R}^n}e^{-c\frac{|x-y|^2}{v}}|f(y)||g(x)|dydx\\
&\leq C&\frac{1+v}{v^{n/2+1}}\int_{\mathbb{R}^n}|f(y)|dy\int_{\mathbb{R}^n}|g(x)|dx<\infty ,\quad v\in (0,\infty ).
\end{eqnarray*}
Hence, the function $\psi(v) =\int_{\mathbb{R}^n}\int_{\mathbb{R}^n} (W_v^{\mathcal{L}}(x,y)-W_v(x,y)) f(y) \  \overline{g(x)} \ dy \ dx, \ \ v\in(0,\infty)$, is differentiable in $(0,\infty)$ and
\begin{equation*}
\frac {d}{dv} \psi(v) =\int_{\mathbb{R}^n}\int_{\mathbb{R}^n} \frac{\partial}{\partial v} (W_v^{\mathcal{L}}(x,y)-W_v(x,y)) f(y) \ \overline{g(x)} \ dy \ dx, \ \ v\in(0,\infty).
\end{equation*}
We can write
\begin{eqnarray} \label {des2.0}
\left\langle T_m^{\mathcal{L}} (f), g \right\rangle &=& \int_0^\infty \phi(v) \left(-\frac {d}{dv} \right)\int_{\mathbb{R}^n}\int_{\mathbb{R}^n} W_v(x,y) f(y) \ dy  \ \overline{g(x)} \ dx \ dv \nonumber \\
&& + \int_0^\infty \phi(v)  \int_{\mathbb{R}^n}\int_{\mathbb{R}^n}
\left(-\frac {\partial}{\partial v}\right)
\left(W_v^{\mathcal{L}}(x,y)-W_v(x,y)\right)f(y) \  \overline{g(x)} \ dy \ dx \ dv.
\end{eqnarray}
Also, we have that
\begin{equation} \label{des2.1}
  \int_{\mathbb{R}^n}\int_{\mathbb{R}^n}\int_0^\infty | \phi(v)|
\left|\frac {\partial}{\partial v}
\left(W_v^{\mathcal{L}}(x,y)-W_v(x,y)\right)\right| |f(y)|\  |\overline{g(x)}| \ dv \ dx \ dy<\infty,
\end{equation}
and
\begin{equation} \label{des2.2}
  \int_{\mathbb{R}^n}\int_0^\infty
\left|\frac {\partial}{\partial v}
\left(W_v^{\mathcal{L}}(x,y)-W_v(x,y)\right)\right| |f(y)| \ | \phi(v)| \ dv \ dy < \infty.
\end{equation}
Indeed, to see (\ref{des2.2}) we write
\begin{eqnarray*}
\lefteqn{\int_{\mathbb{R}^n}\int_0^\infty
\left|\frac {\partial}{\partial v}
\left(W_v^{\mathcal{L}}(x,y)-W_v(x,y)\right)\right| |f(y)|| \phi(v)|  \ dv \ dy} \\
&\leq& || \phi||_\infty\left(\int_{\mathbb{R}^n}\int_0^{\rho(x)^2} +\int_{\mathbb{R}^n}\int_{\rho(x)^2}^\infty \right)
\left|\frac {\partial}{\partial v}
\left(W_v^{\mathcal{L}}(x,y)-W_v(x,y)\right)\right| |f(y)| \ dv  \ dy  \\
&=& B_1(x) +B_2(x).
\end{eqnarray*}
According to \cite[(2.7)]{DGMTZ} we get
\begin{equation} \label{des2.3}
B_2(x) \leq C \int_{\mathbb{R}^n}\int_{\rho(x)^2}^\infty  \frac{e^{-c\frac{|x-y|^2}{v}}}{v^{\frac n2 +1}} |f(y)| \  dv \ dy \leq \frac{C}{\rho (x)^n}, \ \ x\in \mathbb{R}^n.
\end{equation}
Since $0<\rho(x)<\infty$, it follows that $B_2(x)<\infty$, $x\in \mathbb{R}^n$.\\

By proceeding as in \cite[p. 15 -17]{BFHR} we can obtain that
\begin{equation} \label{des2.4}
B_1(x) \leq C ||f||_\infty, \ x \in \mathbb{R}^n.
\end{equation}
Thus, we have proved (\ref{des2.2}). Moreover, by using \cite[Lemma 1.4]{Sh1}, (\ref{des2.3}) and (\ref{des2.4}) imply also (\ref{des2.1}). Then (\ref{des2.0}) can be written
\begin{eqnarray} \label {des2.5}
\left\langle T_m^{\mathcal{L}} (f), g \right\rangle &=& \int_0^\infty \phi(v) \left(-\frac {d}{dv} \right)\int_{\mathbb{R}^n}\int_{\mathbb{R}^n} W_v(x,y) f(y) \ dy  \ \overline{g(x)} \ dx \ dv \nonumber \\
&+&  \int_{\mathbb{R}^n} \left( \lim_{\varepsilon\rightarrow 0^+}  \int_{|x-y|>\varepsilon}
\left(K_\phi^{\mathcal{L}}(x,y)-K_\phi(x,y)\right)f(y) \ dy \right)  \overline{g(x)} \ dx,
\end{eqnarray}
where
\begin{equation*}
K_\phi^{\mathcal{L}} (x,y) = -\int_0^\infty \phi(v)\frac{\partial}{\partial v} W_v^{\mathcal{L}}(x,y) \  \  dv, \ \  x,y\in \mathbb{R}^n, x\neq y,
\end{equation*}
and
\begin{equation*}
K_\phi (x,y) = -\int_0^\infty \phi(v)\frac{\partial}{\partial v} W_v (x,y) \  \ dv, \ \  x,y\in \mathbb{R}^n, x\neq y.
\end{equation*}
On the other hand, as above we can see that
\begin{equation}\label{Lapmul}
\left\langle T_m (f), g \right\rangle = \int_0^\infty \phi(v) \left(-\frac{d}{d v} \right) \int_{\mathbb{R}^n}\int_{\mathbb{R}^n} W_v(x,y) \ f(y) \ dy \ \overline{g(x)} \ dx \ dv.
\end{equation}
where $T_m$ represents the spectral multiplier associated with
$-\Delta$ defined by $m$. Moreover, we can write
\begin{equation} \label{des2.6}
T_m (f)(x) =\lim_{\varepsilon\rightarrow 0^+} \left( \alpha(\varepsilon) \ f(x) + \int_{|x-y|>\varepsilon} K_\phi (x,y) \ f(y) \ dy \right), \,\,a.e. \ x \in \mathbb{R}^n.
\end{equation}
where
$$
\alpha(\varepsilon)=\frac{1}{\Gamma(\frac{n}{2})}\int_0^\infty\phi\Big(\frac{\varepsilon^2}{4u}\Big)e^{-u}u^{\frac{n}{2}-1}du,
\;\;\varepsilon >0.
$$
Also, if there exists the limit $\lim_{t\rightarrow 0^+} \phi(t) = \phi(0^+)$, then $\lim_{\varepsilon \rightarrow 0^+} \alpha(\varepsilon)= \phi(0^+)$,  and
\begin{equation} \label{des2.7}
T_m(f)(x) =\phi(0^+)f(x) + \lim_{\varepsilon \rightarrow 0^+} \int_{|x-y|>\varepsilon} K_\phi (x,y) \ f(y) \ dy, \ \ a.e. \ x \in \mathbb{R}^n.
\end{equation}
Although we are sure that the properties (\ref{des2.6}) and (\ref{des2.7}) are known, we include in the appendix complete proofs for these properties of $T_m$, for the sake the interested reader.

By combining (\ref{des2.5}), (\ref{Lapmul}), (\ref{des2.6}) and
(\ref{des2.7}) we obtain that
\begin{equation*}
T_m^{\mathcal{L}} (f)(x) =\lim_{\varepsilon \rightarrow 0^+} \left( \alpha(\varepsilon) \ f(x) + \int_{|x-y|>\varepsilon} K_\phi^{\mathcal{L}} (x,y) \ f(y) \ dy\right), \ \ a.e. \ x\in \mathbb{R}^n,
\end{equation*}
and
\begin{equation*}
T_m^{\mathcal{L}} (f)(x) =  \phi(0^+) \ f(x) + \lim_{\varepsilon \rightarrow 0^+} \int_{|x-y|>\varepsilon} K_\phi^{\mathcal{L}} (x,y) \ f(y) \ dy, \ \ a.e. \ x\in \mathbb{R}^n,
\end{equation*}
provided that there exists the limit $ \phi(0^+)= \lim_{t\rightarrow 0^+} \phi(t)$.

\section{Proof of Theorems \ref{Lp} and \ref{UMD}}

In this section we present a proof for Theorems \ref{Lp} and \ref{UMD}. Firstly we prove Theorem \ref{UMD}.

\subsection{Proof of Theorem \ref{UMD}}

Guerre-Delabriere \cite[Theorem, p. 402]{GD} established that a
Banach space $B$ is UMD if, only if, for every $\gamma\in
\mathbb{R}$, the operator $\Big(-\frac{d^2}{dx^2}\Big)^{i\gamma}$
can be extended to $L^p(\mathbb{R},B)$ into itself, for some
(equivalently, for any) $1<p<\infty$. In the proof of \cite[Theorem,
p. 402]{GD} a vector valued version of a classical transference
result was used.

Assume that $B$ is a Banach space and $\gamma\in \mathbb{R}$. The operator $\Big(-\frac{d^2}{dx^2}\Big)^{i\gamma}$ takes the form
$$
\Big(-\frac{d^2}{dx^2}\Big)^{i\gamma}f=(|y|^{2i\gamma}\hat{f})\check{}, \,\,\,f\in L^2(\mathbb{R}),
$$
where $\hat f$ denotes the Fourier transform of $f$ and $\check{f}$ the inverse Fourier transform of $f$. If $f\in L^1(\mathbb R)$ we define
$$
\hat f(y)=\int_{\mathbb R} e^{-ixy}f(x)\,dx,\;\; y \in \mathbb R,
$$
and
$$
 \check {f}(y)= \frac{1}{2\pi}\int_{\mathbb R} e^{ixy}f(x)\,dx,\;\; y \in \mathbb R.
$$
As it is well known the Fourier transform can be extended from
$L^1(\mathbb{R})\cap L^2(\mathbb{R})$ to $L^2(\mathbb{R})$ as a
bijective bounded operator from $L^2(\mathbb{R})$ into itself. The
operator $\Big(-\frac{d^2}{dx^2}\Big)^{i\gamma}$ is bounded from
$L^p(\mathbb{R})$ into itself, for every $1<p<\infty$. If
$1<p<\infty$ and $f\in L^p(\mathbb{R})\otimes B$, that is,
$f=\sum_{j=1}^r \beta_jf_j$, where $\beta_j\in B$, $f_j\in
L^p(\mathbb{R})$, $j=1,\ldots,r\in \mathbb{N}$, we define, as usual,
$$
\Big(-\frac{d^2}{dx^2}\Big)^{i\gamma}(f)=\sum_{j=1}^r\beta_j\Big(-\frac{d^2}{dx^2}\Big)^{i\gamma}(f_j).
$$

We also consider the operator $\Big(-\frac{d^2}{dx^2}\Big)^{i\gamma}_{|\mathbb{T}}$, where $\mathbb{T}=[0,2\pi)$ denotes the one-dimensional torus, defined by
$$
\Big(-\frac{d^2}{dx^2}\Big)^{i\gamma}_{|\mathbb{T}}(g)(x)=\sum_{j\in \mathbb{Z},\,j\neq 0}|j|^{i\gamma}c_j(g)e^{ijx},\,\, x \in (0,2\pi)\,\,\, and \,\,\,g\in L^p(\mathbb{T}),\,\,1< p<\infty,
$$
being $\displaystyle
c_j(g)=\frac{1}{2\pi}\int_0^{2\pi}g(\theta)e^{-ij\theta}d\theta$,
$j\in \mathbb{Z}$. The operator
$\Big(-\frac{d^2}{dx^2}\Big)^{i\gamma}_{|\mathbb{T}}$ is bounded
from $L^p(\mathbb{T})$ into itself, $1<p<\infty$. If $1<p<\infty$
and $g\in L^p(\mathbb{T})\otimes B$, that is, $g=\sum_{j=1}^r
\beta_jg_j$, where $\beta_j\in B$, $g_j\in L^p(\mathbb{T})$,
$j=1,\ldots,r\in \mathbb{N}$, we define
$$
\Big(-\frac{d^2}{dx^2}\Big)^{i\gamma}_{|\mathbb{T}}(g)=\sum_{j=1}^r\beta_j\Big(-\frac{d^2}{dx^2}\Big)^{i\gamma}_{|\mathbb{T}}(g_j).
$$

Guerre-Delabriere (\cite[p. 402]{GD}) showed that if
$\Big(-\frac{d^2}{dx^2}\Big)^{i\gamma}_{|\mathbb{T}}$ can be
extended to $L^2(\mathbb{T},B)$ as a bounded operator from
$L^2(\mathbb{T},B)$ into itself, then $B$ is UMD. Moreover, she used
a vector valued transference result (see \cite{CW} for the scalar
result) that implies that
$\Big(-\frac{d^2}{dx^2}\Big)^{i\gamma}_{|\mathbb{T}}$ can be
extended to $L^2(\mathbb{T},B)$ as a bounded operator from
$L^2(\mathbb{T},B)$ into itself, provided that
$\Big(-\frac{d^2}{dx^2}\Big)^{i\gamma}$ can be extended to
$L^2(\mathbb{R},B)$ as a bounded operator from $L^2(\mathbb{R},B)$
into itself. Note that, by using vector valued Calder\'on-Zygmund
theory (\cite{RbRT}) we can see that
$\Big(-\frac{d^2}{dx^2}\Big)^{i\gamma}$ can be extended to
$L^p(\mathbb{R},B)$ as a bounded operator from $L^p(\mathbb{R},B)$
into itself, for some $1<p<\infty$, if and only if
$\Big(-\frac{d^2}{dx^2}\Big)^{i\gamma}$ can be extended to
$L^2(\mathbb{R},B)$ as a bounded operator from $L^2(\mathbb{R},B)$
into itself.

In order to prove Theorem \ref{UMD} we need to show a higher
dimension version of Guerre-Delabriere's result. The operators
$(-\Delta)^{i\gamma}$ (respectively, $(-\Delta)^{i\gamma}_{|\mathbb{T}^n}$)
are defined on $L^p(\mathbb{R}^n)$ and $L^p(\mathbb{R}^n)\otimes B$
(respectively, on $L^p(\mathbb{T}^n)$ and $L^p(\mathbb{T}^n)\otimes
B$), $1<p<\infty$, in the natural way.

\begin{propo} \label{GDn} Let $B$ be a Banach space and $n\in \mathbb{N}$. Then, the following assertions are equivalent.

(i) $B$ is UMD.

(ii) For every $\gamma\in \mathbb{R}$ and for some (equivalently, for any) $1<p<\infty$, the operator $(-\Delta)^{i\gamma}$ can be extended to $L^p(\mathbb{R}^n,B)$ as a bounded operator from $L^p(\mathbb{R}^n,B)$ into itself.
\end{propo}

\begin{proof}
$(i)\Rightarrow (ii)$ It is a consequence of \cite[Proposition
3]{Zim}.

$(ii)\Rightarrow (i)$ We show this part by adapting standard
transference arguments to a vector valued setting. For the sake of
completeness we include the proof.

Let $\gamma\in \mathbb{R}$ and $1<p<\infty$. Suppose that the operator $(-\Delta)^{i\gamma}$ can be extended to $L^p(\mathbb{R}^n,B)$ as a bounded operator from $L^p(\mathbb{R}^n,B)$ into itself. We choose an even smooth function on $\mathbb{R}$ such that $\phi(x)=1$, $|x|\le 1/4$, and $\phi(x)=0$, $|x|\ge 1/2$. We split the operator $(-\Delta)^{i\gamma}$ as follows
\begin{eqnarray*}
(-\Delta)^{i\gamma}(f)&=&(\phi(|x|^2)|x|^{2i\gamma}\hat{f})\check{}+ ((1-\phi(|x|^2))|x|^{2i\gamma}\hat{f})\check{}\\
&=&A_1(f)+A_2(f),\,\,\,f\in C_c^\infty(\mathbb{R}^n)\otimes B.
\end{eqnarray*}
Here, $\hat h$ denotes the Fourier transform of $h$ and $\check{h}$
the inverse Fourier transform of $h$ in $\mathbb{R}^n$, defined, for
every  $h \in L^1(\mathbb R^n)$, by
$$
\hat h(y)=\int_{\mathbb R^n} e^{-ixy}h(x)\,dx,\;\; y \in \mathbb R^n,
$$
and
$$
 \check h(y)= \frac{1}{(2\pi)^n}\int_{\mathbb R^n} e^{ixy}h(x)\,dx,\;\; y \in \mathbb R^n.
$$
Also, we consider the function $\varphi(x)=\phi(|x|^2)$, $x\in \mathbb{R}^n$, and the Fourier multiplier $T_\varphi$ defined by
$$
T_\varphi(f)=(\varphi\hat{f})\check{},\,\,\,f\in C_c^\infty(\mathbb{R}^n)\otimes B,
$$
in a natural way. Since $\hat{\varphi}\in L^1(\mathbb{R}^n)$, $T_\varphi$ can be extended to $L^p(\mathbb{R}^n,B)$ as a bounded operator from  $L^p(\mathbb{R}^n,B)$ into itself. Then, $A_1$ and therefore $A_2$ can be extended to $L^p(\mathbb{R}^n,B)$ as a bounded operator from  $L^p(\mathbb{R}^n,B)$ into itself.

We denote by $\mathcal{P}(\mathbb{T}^n,X)$ the space of
trigonometric polynomials of period $2\pi$ on $\mathbb{T}^n$ with
coefficients in a Banach space $X$. Let $P\in
\mathcal{P}(\mathbb{T}^n,B)$ and $Q\in
\mathcal{P}(\mathbb{T}^n,B')$, where $B'$ is the dual space of $B$. Since $A_2$ can be extended to
$L^p(\mathbb{R}^n,B)$ as a bounded operator from
$L^p(\mathbb{R}^n,B)$ into itself, by proceeding as in the proof of
\cite[Theorem 3.8, p. 260]{SWe} we have that
$$
\Big|\int_{\mathbb{T}^n}\langle (-\Delta)^{i\gamma}_{|\mathbb{T}^n}(P)(x),Q(x)\rangle dx\Big|\le C\|P\|_{L^p(\mathbb{T}^n,B)}\|Q\|_{L^{p'}(\mathbb{T}^n,B')},
$$
where $p'$ is the exponent conjugated to $p$.

By using \cite[Lemma 2.3]{GLY} we get
$$
\|(-\Delta)^{i\gamma}_{|\mathbb{T}^n}(P)\|_{L^p(\mathbb{T}^n,B)}\le C\|P\|_{L^p(\mathbb{T}^n,B)}.
$$
Hence, $(-\Delta)^{i\gamma}_{|\mathbb{T}^n}$ can be extended to $L^p(\mathbb{T}^n,B)$ as a bounded operator from  $L^p(\mathbb{T}^n,B)$ into itself.

In order to see that the operator $\Big(-\frac{d^2}{dx^2}\Big)^{i\gamma}_{|\mathbb{T}}$ can be extended to $L^p(\mathbb{T},B)$ as a bounded operator from  $L^p(\mathbb{T},B)$ into itself, it is sufficient to use that $(-\Delta)^{i\gamma}_{|\mathbb{T}^n}$ can be extended to $L^p(\mathbb{T}^n,B)$ as a bounded operator from  $L^p(\mathbb{T}^n,B)$ into itself, and to extend every function $f\in L^p(\mathbb{T})$ to $\mathbb{T}^n$ in the natural way, that is, defining $\tilde{f}(x_1,\ldots,x_n)=f(x_1)$, $(x_1,\ldots,x_n)\in \mathbb{T}^n$.

According to \cite[Theorem, p. 402]{GD} the above arguments allow us
to conclude that $(ii)\Rightarrow (i)$.
\end{proof}

We are going to prove Theorem \ref{UMD}. Let $\gamma \in
\mathbb{R}$. The imaginary power $\mathcal{L}^{i\gamma}$ of
$\mathcal{L}$ (respectively, $(-\Delta)^{i\gamma}$ of  $-\Delta$) is
the spectral multiplier associated with $\mathcal{L}$ (respectively,
$-\Delta$) defined by the function $m_\gamma (\lambda) =
\lambda^{i\gamma},$ $\lambda\in (0,\infty)$. Note that
$m_\gamma(\lambda)= \lambda \int_0^\infty e^{-\lambda t} \phi_\gamma(t) \ dt, \ \lambda\in (0,\infty)$, where $\phi_\gamma (t) = \frac{t^{-i\gamma}}{\Gamma(1-i\gamma)}, \ t\in(0,\infty)$.\\

Assume that $B$ is a Banach space and that $f=\sum_{j=1}^d \beta_j
f_j$, where $f_j \in C_c^\infty(\mathbb{R}^n)$ and $\beta_j\in B$,
$j=1,\ldots, d$. By Theorem \ref{Valorprincipal} and (\ref{des2.6}), we have that
\begin{equation*}
T_{m_\gamma}^{\mathcal{L}} (f) (x) =\sum_{j=1}^d \beta_j \lim_{\varepsilon \rightarrow 0^+}\left(\alpha(\varepsilon) f_j(x) +  \int_{|x-y|>\varepsilon} K_{\phi_\gamma}^{\mathcal{L}}(x,y) f_j(y) dy \right), \ \ a.e. \ x\in \mathbb{R}^n,
\end{equation*}
and
\begin {equation*}
T_{m_\gamma} (f) (x) =\sum_{j=1}^d \beta_j \lim_{\varepsilon \rightarrow 0^+}\left(\alpha(\varepsilon) f_j(x) +  \int_{|x-y|>\varepsilon} K_{\phi_\gamma}(x,y) f_j(y) dy \right), \ \ a.e. \ x\in \mathbb{R}^n.
\end{equation*}
We split the operator $T_{m_\gamma}$ as follow
$$
T_{m_\gamma}=T_{m_\gamma, g}+ T_{m_\gamma,\ell},
$$
where $T_{m_\gamma,g}(f) (x) = \int_{|x-y|\geq \rho(x)} K_{\phi_\gamma}(x,y) \ f(y) \ dy, \ x\in\mathbb{R}^n$.

The operator $S_{m_\gamma} = T_{m_\gamma}^{\mathcal{L}} -
T_{m_\gamma,\ell}$ can be extended to $L^p(\mathbb{R}^n,B)$ as a
bounded operator from $L^p(\mathbb{R}^n,B)$ into itself, for every
$1<p<\infty$. Indeed, we can write
\begin{eqnarray*}
S_{m_\gamma} (f)(x) &=& \lim_{\varepsilon\rightarrow 0^+} \int_{\varepsilon <|x-y|<\rho(x)} f(y) \left(K_{\phi_\gamma}^{\mathcal{L}}(x,y) -K_{\phi_\gamma}(x,y)\right) \ dy\\
&&-\int_{|x-y|\geq \rho(x)} f(y) \int_0^\infty \phi_\gamma(t) \frac{\partial}{\partial t} W_t^{\mathcal{L}}(x,y) \ dt \ dy\\
&=& S_{m_\gamma,1} (f)(x)+ S_{m_\gamma,2} (f)(x).
\end{eqnarray*}
By proceeding as in \cite[p. 15-17]{BFHR} we can get
$$||S_{m_\gamma,1} (f)(x)||_B \leq \int_{|x-y|<\rho(x)} ||f(y)||_B \ \left|K_{\phi_\gamma}^{\mathcal{L}}(x,y)-K_{\phi_\gamma}(x,y)\right|dy \leq C \mathcal{M}(||f||) (x),\,\,\,x\in \mathcal{R}^n,$$
and
\begin{eqnarray*}
&&||S_{m_\gamma,2} f(x)||_B \leq \int_{|x-y|\geq \rho(x)} ||f(y)||_B \int_0^{\rho(x)^2} |\phi_\gamma(t) | \left|\frac{\partial}{\partial t} W_t^{\mathcal{L}}(x,y)\right|\ dt \ dy\\
 &+&
\int_{|x-y|\geq \rho(x)} ||f(y)||_B \int_{\rho(x)^2}^\infty |\phi_\gamma(t) | \left|\frac{\partial}{\partial t} W_t^{\mathcal{L}}(x,y)\right|\ dt \ dy \leq C \mathcal{M}(||f||) (x),\,\,\,x\in \mathbb{R}^n,
\end{eqnarray*}
because $||\phi_\gamma||_\infty=1$.

Hence, by using the well known Maximal Theorem we conclude that the operator $S_{m_\gamma}$ can be extended to $L^p(\mathbb{R}^n,B)$ as a bounded operator from $L^p(\mathbb{R}^n,B)$ into itself, for every $1<p <\infty$.

Suppose now that $B$ is a $UMD$ Banach space. According to Proposition \ref{GDn} the operator $T_{m_\gamma}$ can be extended to $L^p(\mathbb{R}^n,B)$ as a bounded operator from $L^p(\mathbb{R}^n,B)$ into itself, for every $1<p<\infty$.

For every $f\in L_c^\infty(\mathbb{R}^n) \otimes B$ we have that
$$T_{m_\gamma} (f)(x) = \int_{\mathbb{R}^n} K_{\phi_\gamma}(x,y) f(y) dy, \ \ a.e. \ x\notin supp \ f.$$

Moreover, $K_{\phi_\gamma}$ is a standard Calderón-Zygmund kernel, that is, there exists $C>0$ such that\\

$$
\displaystyle {|K_{\phi_\gamma}(x,y)| \leq \frac{C}{|x-y|^n}, \ \ x\neq y,}
$$
and
$$ \displaystyle {\sum_{j=1}^n \left(\left|\frac{\partial K_{\phi_\gamma}(x,y)}{\partial x_j}\right|+\left|\frac{\partial K_{\phi_\gamma}(x,y)}{\partial y_j}\right| \right)\leq \frac{C}{|x-y|^{n+1}}, \ \ x\neq y.}
$$
Then, by proceeding as in the scalar case (see \cite[p. 34]{Stein}),
we can show that the maximal operator
$$T_{m_\gamma}^*(f)(x) = \sup_{\varepsilon>0}\left\| \int_{|x-y|>\varepsilon} K_{\phi_\gamma} (x,y) \ f(y) \ dy \right\|_B,$$
is bounded from $L^p(\mathbb{R}^n,B)$ into $L^p(\mathbb{R}^n)$, for every $1<p<\infty$.

It is clear that $||T_{m_\gamma, g}(f)(x)||_B \leq T_{m_\gamma}^* (f) (x), \ x\in\mathbb{R}^n.$ Then, for every $1<p<\infty$, $T_{m_\gamma, g}$ is bounded from $L^p(\mathbb{R}^n,B)$ into itself.

Hence, since $T_{m_\gamma, \ell}= T_{m_\gamma}-T_{m_\gamma, g}$, we conclude that, for every $1<p<\infty$, $T_{m_\gamma, \ell} $, and then also $T_{m_\gamma}^{\mathcal{L}}$, are bounded from $L^p(\mathbb{R}^n,B)$ into itself.

Assume now that for a certain $1<p<\infty$ and every $\gamma \in\mathbb{R}$ the operator $T_{m_\gamma}^{\mathcal{L}}$ can be extended to
$L^p(\mathbb{R}^n,B)$ as a bounded operator from $L^p(\mathbb{R}^n,B)$ into itself.
Then, for every  $\gamma\in \mathbb{R}$ the operator $T_{m_\gamma,\ell}$ can be extended to $L^p(\mathbb{R}^n,B)$ as a bounded operator from $L^p(\mathbb{R}^n,B)$ into itself.
According to Proposition \ref{GDn} in order to show that $B$ is a $UMD$ Banach space it is sufficient to show that for every  $\gamma\in\mathbb{R}$, $T_{m_\gamma}$ can be extended to $L^p(\mathbb{R}^n,B)$ as a bounded operator from $L^p(\mathbb{R}^n,B)$ into itself.

Let $\gamma \in\mathbb{R}$. Suppose that $f\in C_c^\infty(\mathbb{R}^n)$ and $\mbox{supp} \ f \subset B(0,M)$ for a certain $M>0$. For every $R>0$ we define $f_R(x)= f(\sqrt R x), \ x\in \mathbb{R}^n$. It is clear that, $\mbox{supp} \ f_R \subset B(0,\frac{M}{\sqrt R}), \ R>0$.

In the following our arguments  are inspired in the ones developed by Abu-Falahah, Stinga and Torrea in \cite{AST}. We are going to show that for every $\lambda >0$ there exists $R>0$ such that $\mbox{supp} \ f_R \subset B(\frac x R, \rho(\frac xR)),$ provided that $|x|<\lambda$. According to \cite[Lemma 1.1]{Sh1} there exists $C_1>0$ for which
$$\frac {1}{C_1} \rho(y) \leq \rho(x)\leq C_1 \rho(y), \ \ |x-y|\leq \rho(x).$$
Let $\lambda >0$. From \cite[Lemma 3.5]{AST} we can find $R_\lambda>0$ such that  $|y-\frac x R| <\rho(\frac xR)$, when $|y|<\frac{C_1^2 \rho(0)}{2}, |x|<\lambda$ and $R\geq R_\lambda$. We can take $R\geq \max \{ R_\lambda, \left(\frac{2M}{\rho(0)C_1^2}\right)^2\}$. Then $\mbox{supp} \ f_R \subseteq B(\frac xR, \rho(\frac x R)), \ |x|<\lambda.$

We can write, for every $R>0$,
\begin{eqnarray*}
\lefteqn{ T_{m_\gamma}(f_R) \left(\frac{x}{\sqrt R}\right)} \\ &=& \lim_{\varepsilon \rightarrow 0^+}\left(\alpha(\varepsilon) f_R\left (\frac{x}{\sqrt R}\right) +  \int_{|\frac{x}{\sqrt R}-y|>\varepsilon} f_R(y) \int_0^\infty \phi_\gamma (t) \left(-\frac{\partial}{\partial t}\right)W_t\left(\frac{x}{\sqrt R},y\right) \ dt \ dy\right)\\
&=& \lim_{\varepsilon \rightarrow 0^+}\left(\alpha(\varepsilon) f(x) +  \int_{|x-u|>\varepsilon\sqrt R} f(u) \int_0^\infty \phi_\gamma (t) \left(-\frac{\partial}{\partial t}\right)W_t\left(\frac{x}{\sqrt R},\frac{u}{\sqrt R}\right) \ dt \  \frac {du}{R^{n/2}}\right)\\
&=& \lim_{\varepsilon \rightarrow 0^+}\left(\alpha(\varepsilon) f(x) +  R \int_{|x-u|>\varepsilon\sqrt R} f(u) \int_0^\infty \phi_\gamma (t) \left.\left(-\frac{\partial}{\partial s}W_s(x,u)\right)\right|_{s=Rt} \ dt \ du\right) \\
&=& \lim_{\varepsilon \rightarrow 0^+}\left(\alpha(\varepsilon) f(x) +  \int_{|x-u|>\varepsilon\sqrt R} f(u) \int_0^\infty \phi_\gamma \left(\frac{s}{R}\right) \left(-\frac{\partial}{\partial s}\right)W_s(x,u) \ ds \  du\right), \ \ a.e \ x\in \mathbb{R}^n.\\
\end{eqnarray*}
Here $\displaystyle\alpha(\varepsilon)= \frac{1}{\Gamma(\frac{n}{2})}\int_0^\infty e^{-u}u^{\frac{n}{2}-1}\phi_\gamma\left(\frac{\varepsilon^2}{4u}\right)du, \;\;\varepsilon\in (0,1).$

Since $\phi_\gamma(as)=a^{-i\gamma}\phi_\gamma(s)$, $a,s>0$, it follows that, for
every $R>0$,
\begin{eqnarray*}
T_{m_\gamma}(f_R)\left(\frac{x}{\sqrt R}\right) &=& R^{i\gamma}\lim_{\varepsilon \rightarrow 0^+}\left(\alpha(\varepsilon\sqrt R)f(x) + \int_{|x-u|>\varepsilon \sqrt R}f(u)\int_0^\infty \phi_\gamma(s) \left(-\frac{\partial}{\partial s}\right)W_s(x,u)ds du\right)\\
&=& R^{i\gamma}T_{m_\gamma}(f)(x),\;\;\mbox{a.e.}\;x\in \mathbb R^n.
\end{eqnarray*}

As it was proved above, for every $N\in \mathbb{N}$, there exists $R_N>0$ such that $\mbox{supp} \ f_{R_N} \subseteq B \left(\frac{x}{R_N},\rho(\frac{x}{R_N})\right), \ |x| \leq N$, and $R_N \leq R_{N+1}$.
Then, it follows that
\begin{eqnarray*}
T_{m_\gamma} (f) (x) &=& R_N^{-i\gamma}  T_{m_\gamma} (f_{R_N} \chi_{B(\frac{x}{R_N},\rho(\frac{x}{R_N}))}) (\frac{x}{\sqrt{R_N}} )\\
&=& R_N^{-i\gamma}  T_{m_\gamma,\ell} (f_{R_N}) (\frac{x}{\sqrt{R_N}}), \ |x|\leq N, \ N \in \mathbb{N}.
\end{eqnarray*}
We deduce that,
\begin{eqnarray*}
\int_{B(0,N)} |T_{m_\gamma}(f)(x)|^p dx &\leq &  R_N^{n/2}\int_{\mathbb{R}^n} |T_{m_\gamma,\ell}(f_{R_N})(x)|^p dx \\
&\leq& C R_N^{n/2}\int_{\mathbb{R}^n}|f_{R_N}(x)|^p dx \\
&\leq & C ||f||_p^p, \ \ N \in \mathbb{N}.
\end{eqnarray*}
Note that $C$ does not depend on $N$.

We conclude that
\begin{equation}\label{acomul}
||T_{m_\gamma} (f)||_p \leq C ||f||_p.
\end{equation}
Also, (\ref{acomul}) holds for every $f \in C_c^\infty(\mathbb
R^n)\otimes B$. Hence, $T_{m_\gamma}$ can be extended to
$L^p(\mathbb{R}^n,B)$ as a bounded operator from
$L^p(\mathbb{R}^n,B)$ into itself. The proof is finished.

\subsection{Proof of Theorem \ref{Lp}}
\begin{proof}
This proof follows the same way that the one of the $L^p$-boundedness of the imaginary power $\mathcal{L}^{i\gamma}$ of $\mathcal{L}$,  $\gamma \in \mathbb R^n$, when $B$ is a UMD space.

Suppose that $m(\lambda)=\lambda \int_0^\infty e^{-\lambda t}\phi(t)dt$, $\lambda\in (0,\infty)$, where $\phi \in L^\infty(0,\infty)$. Let $f\in C_c^\infty(\mathbb R^n)$. According to Theorem \ref{Valorprincipal}
$$
T_m^{\mathcal{L}}(f)(x)=\lim_{\varepsilon \rightarrow 0^+}\left(\alpha(\varepsilon)f(x)+\int_{|x-y|>\varepsilon}f(y)K_\phi^{\mathcal{L}}(x,y)dy\right),\;\;\mbox{a.e.} \, x \in \mathbb R^n.
$$
Also, by (\ref{des2.6}),
$$
T_m(f)(x)=\lim_{\varepsilon \rightarrow 0^+}\left(\alpha(\varepsilon)f(x)+\int_{|x-y|>\varepsilon}f(y)K_\phi(x,y)dy\right),\;\;\mbox{a.e.} \, x \in \mathbb R^n.
$$
Here $\alpha \in L^\infty(0,\infty)$.

The operator $T_m$ is bounded from $L^p(\mathbb R^n)$ into itself, for every $1<p<\infty$, and from $L^1(\mathbb R^n)$ into $L^{1,\infty}(\mathbb R^n)$. Moreover $T_m$ is a Calder\'on-Zygmund operator. Hence, the maximal operator $T^*_m$ defined by
$$
T_m^*(f)(x)=\sup_{\varepsilon >0}\left|\int_{|x-y|>\varepsilon}f(y)K_\phi(x,y)dy\right|
$$
is bounded from $L^p(\mathbb R^n)$ into itself, for every $1<p<\infty$, and from $L^1(\mathbb R^n)$ into $L^{1,\infty}(\mathbb R^n)$. Also, the same $L^p$-boundedness properties are satisfied by the operators
$$
T_{m,\ell}(f)(x)=\lim_{\varepsilon \rightarrow 0}\left(\alpha(\varepsilon)f(x)+\int_{\varepsilon < |x-y| < \rho(x)}f(y)K_\phi(x,y)dy\right)
$$
and
$$
T_{m,g}(f)(x)=\int_{|x-y| \geq \rho(x)}f(y)K_\phi(x,y)dy.
$$
The difference $T_m^{\mathcal{L}}(f)-T_{m,\ell}(f)$ can be written as
$$
T_m^{\mathcal{L}}(f)(x)-T_{m,\ell}(f)(x)=\int_{|x-y|<\rho(x)} (K_\phi^{\mathcal{L}}(x,y)-K_\phi(x,y))f(y)dy + \int_{|x-y|\geq\rho(x)} K_\phi^{\mathcal{L}}(x,y)f(y)dy.
$$
By proceeding as in the proof of Theorem \ref{UMD} we can see that the operator $T_m^{\mathcal{L}}- T_{m,\ell}$ is bounded from $L^p(\mathbb R^n)$ into itself, for every $1<p<\infty$, and from $L^1(\mathbb R^n)$ into $L^{1,\infty}(\mathbb R^n)$.

Hence we conclude that $T_m^{\mathcal{L}}$ can be extended to $L^p(\mathbb R^n)$, $1<p<\infty$, as a bounded operator from $L^p(\mathbb R^n)$ into itself, for every $1<p<\infty$, and from $L^1(\mathbb R^n)$ into $L^{1,\infty}(\mathbb R^n)$.

Moreover, we can deduce that the maximal operator
$$
T_m^{\mathcal{L},*}(f)(x)=\sup_{\varepsilon >0}\left|\int_{|x-y|>\varepsilon}K_\phi^{\mathcal{L}}(x,y)f(y)dy\right|,\,\, x \in \mathbb R^n,
$$
is bounded from $L^p(\mathbb R^n)$ into itself, for every $1<p<\infty$, and from $L^1(\mathbb R^n)$ into $L^{1,\infty}(\mathbb R^n)$.

Hence, for every $f \in L^p(\mathbb R^n)$, $1 \leq p < \infty$, there exists the limit
$$
\lim_{\varepsilon \rightarrow 0^+}\left(f(x)\alpha(\varepsilon) + \int_{|x-y|>\varepsilon}f(y)K_\phi^{\mathcal{L}}(x,y)dy\right),\;\; \mbox{a.e.}\  x \in \mathbb R^n,
$$
and, for every $f \in L^2(\mathbb R^n)$,
$$
T_m^{\mathcal{L}}(f)(x)= \lim_{\varepsilon \rightarrow 0^+}\left(f(x)\alpha(\varepsilon) + \int_{|x-y|>\varepsilon}f(y)K_\phi^{\mathcal{L}}(x,y)dy\right),\;\; \mbox{a.e.} \ x \in \mathbb R^n.
$$
We conclude that the operator $T_m^{\mathcal{L}}$ can be extended from $L^2(\mathbb R^n) \cap L^p(\mathbb R^n)$ to $L^p(\mathbb R^n)$, $1 \leq p < \infty$, as a bounded operator from $L^p(\mathbb R^n)$ into itself, for every $1<p<\infty$ and from $L^1(\mathbb R^n)$ into $L^{1,\infty}(\mathbb R^n)$.
\end{proof}

\section{Appendix}
In this section we present a pointwise representation of the multiplier $T_m$. We establish the properties (\ref{des2.6}) and (\ref{des2.7}).
\begin{proof}
For every $f\in L^2(\mathbb R^n)$ we have that
$$
T_m(f)=(m(|y|^2)\hat f)\check{}.
$$
%where $\hat f$ denotes the Fourier transform of $f$ and $\check{g}$ the inverse Fourier transform of $g$. If $h \in L^1(\mathbb R^n)$ we define
%$$
%\hat h(y)=\int_{\mathbb R^n} e^{-ixy}h(x)\,dx,\;\; y \in \mathbb R^n,
%$$
%and
%$$
% \check h(y)= \frac{1}{(2\pi)^n}\int_{\mathbb R^n} e^{ixy}h(x)\,dx,\;\; y \in \mathbb R^n.
%$$
Let $f \in C_c^\infty(\mathbb R^n)$. We can write
\begin{eqnarray*}
T_m(f)(x) &=& \frac{1}{(2\pi)^n}\int_{\mathbb R^n}e^{ixy}m(|y|^2)\hat f(y)dy\\
&=& \frac{1}{(2\pi)^n}\int_{\mathbb R^n}e^{ixy}\hat f(y)|y|^2\int_0^\infty\phi(t)e^{-t|y|^2}dtdy\\
&=&\frac{1}{(2\pi)^n}\int_0^\infty\phi(t)\int_{\mathbb R^n}e^{ixy}|y|^2e^{-t|y|^2}\hat f(y)dy dt,\;\;x\in\mathbb R^n.
\end{eqnarray*}
The interchange in the order of integration is justified because
$$
\int_{\mathbb R^n} |\hat f(y)||y|^2\int_0^\infty e^{-t|y|^2} |\phi(t)|dtdy \leq \|\phi\|_\infty \int_{\mathbb R^n}|\hat f(y)|dy < \infty.
$$
Then,
\begin{eqnarray*}
T_m(f)(x) &=& \frac{1}{(2\pi)^n}\int_0^\infty\phi(t)\int_{\mathbb R^n} e^{ixy}e^{-t|y|^2}\widehat{ (-\Delta)f(y)}dydt\\
&=&\frac{-1}{(2\pi)^n}\int_0^\infty\phi(t)\int_{\mathbb R^n}\Delta f(z)\int_{\mathbb R^n}e^{-iy(z-x)}e^{-t|y|^2}dydzdt\\
&=& -\int_0^\infty\phi(t)\int_{\mathbb R^n}\Delta f(z)W_t(x,z) dzdt,\;\; x \in \mathbb R^n.
\end{eqnarray*}
We have taken into account that
$$
\int_{\mathbb R^n}\int_{\mathbb R^n}|\Delta f(z)|e^{-t|y|^2}dzdy < \infty,\;\; t>0,
$$
and that
$$
\int_{\mathbb R^n}e^{-iyz}e^{-t|y|^2}dy =\left(\frac{\pi}{t}\right)^{\frac{n}{2}}e^{-\frac{|z|^2}{4t}}.
$$
Since $\displaystyle\int_{\mathbb R^n}\Delta f(z)dz=\widehat{\Delta f}(0)=-|y|^2\hat{f}(y)_{|y=0}=0$, we can write
$$
T_m(f)(x)= -\int_0^\infty\phi(t)\int_{\mathbb R^n}\Delta f(z)\left (W_t(x,z)-\frac{\chi_{(1,\infty)}(t)}{(4\pi t)^{\frac{n}{2}}}\right)dzdt,\;\; x \in \mathbb R^n.
$$
It is not hard to see that
$$
\left|W_t(x,z)-\frac{1}{(4\pi t)^{\frac{n}{2}}}\right| \leq C\frac{|x-z|^2}{t^{\frac{n+2}{2}}},\;\;x,z\in \mathbb R^n\;\;\hbox{and}\;\;t>0.
$$
Hence it follows that
\begin{multline*}
\shoveleft{\int_0^\infty\int_{\mathbb R^n}|\Delta f(z)|\left|W_t(x,z)-\frac{\chi_{(1,\infty)}(t)}{(4\pi t)^{\frac{n}{2}}}\right|dzdt}\\
\shoveleft{\hspace{35mm}\leq C\left(\int_0^1\int_{\mathbb R^n}W_t(x,z)dzdt + \int_1^\infty\int_{\hbox{supp} f}\frac{|x-z|^2}{t^{\frac{n+2}{2}}}dz dt\right)}\\
\leq C(1+|x|^2),\;\;x \in \mathbb R^n.\hspace{68mm}
\end{multline*}
Then,
$$
T_m(f)(x)=-\lim_{\varepsilon \rightarrow 0^+}\int_0^\infty\phi(t)\int_{|x-z|>\varepsilon}\Delta f(z) \left(W_t(x,z) - \frac{\chi_{(1,\infty)}(t)}{(4\pi t)^{\frac{n}{2}}}\right)dzdt,\;\;x\in \mathbb R^n.
$$
Let $0<\varepsilon <1$. The Green formula leads to,
\begin{eqnarray*}
\lefteqn{\int_{|x-z|>\varepsilon}\Delta f(z)\left(W_t(x,z)-\frac{\chi_{(1,\infty)}(t)}{(4\pi t)^{\frac{n}{2}}}\right)dz}\\
&=& \int_{|x-z|>\varepsilon}f(z)\Delta_z W_t(x,z)dz + \int_{|x-z|=\varepsilon}\partial_n f(z)\left(W_t(x,z)-\frac{\chi_{(1,\infty)}(t)}{(4\pi t)^{\frac{n}{2}}}\right)d\sigma(z) \\
&&- \int_{|x-z|=\varepsilon}f(z) \partial_{n,z}W_t(x,z)d\sigma(z),\;\;x \in \mathbb R^n\;\; \mbox{and}\;\; t>0.
\end{eqnarray*}
Here $\partial_n$ represents the derivative in the direction normal exterior to the sphere $S_\varepsilon=\{z\in \mathbb R^n:|z-x|=\varepsilon\}$.

By using \cite[Lemma 2.1]{StTo} we have that
\begin{eqnarray*}
\lefteqn{\left|\int_0^\infty \phi(t)\int_{|x-z|=\varepsilon} \partial_nf(z)\left(W_t(x,z)-\frac{\chi_{(1,\infty)}(t)}{(4\pi t)^{\frac{n}{2}}}\right)d\sigma(z)dt\right|}\\
&\leq&C\int_{|x-z|=\varepsilon}\left(\int_0^1\frac{e^{-\frac{|x-z|^2}{4t}}}{t^{\frac{n}{2}}}dt+\int_1^\infty\frac{|x-z|^2}{t^{\frac{n}{2}+1}}dt\right)d\sigma(z)\\
&\leq&C\int_{|x-z|=\varepsilon}\left(\frac{1}{|x-z|^{n-2}}+|x-z|^2\right)d\sigma(z) \leq C\varepsilon,\;\;x\in \mathbb R^n.
\end{eqnarray*}
If $n(z)$ denotes a unitary vector in the direction exterior normal  in $z \in S_\varepsilon$, we obtain
\begin{eqnarray*}
\partial_{n,z}W_t(x,z) &=& \langle \nabla_zW_t(x,z),n(z)\rangle = W_t(x,z)\langle\frac{x-z}{2t},n(z)\rangle \\
&=& W_t(x,z)\frac{|x-z|}{2t} = \frac{e^{-\frac{\varepsilon^2}{4t}}\varepsilon}{2(4\pi)^{\frac{n}{2}}t^{\frac{n}{2}+1}},\;\; z \in S_\varepsilon.
\end{eqnarray*}

Moreover, $\displaystyle\sigma(S_\varepsilon)= 2\varepsilon^{n-1}\frac{\pi^{\frac{n}{2}}}{\Gamma(\frac{n}{2})}$. Then we have that
\begin{eqnarray*}
\lefteqn{\int_0^\infty\int_{|x-z|=\varepsilon}f(z)\partial_{n,z}W_t(x,z)d\sigma(z)\phi(t)dt} \\
&=& \varepsilon\int_{|x-z|=\varepsilon}f(z)\int_0^\infty \frac{e^{-\frac{\varepsilon^2}{4t}}}{2(4\pi)^{\frac{n}{2}}}\frac{\phi(t)}{t^{\frac{n}{2}+1}}dtd\sigma(z)\\
&=& \frac{1}{\varepsilon^{n-1}2\pi^{\frac{n}{2}}}\int_{|x-z|=\varepsilon}f(z)\int_0^\infty\phi(\frac{\varepsilon^2}{4u})e^{-u}u^{\frac{n}{2}-1}dud\sigma(z)\\
&=&\frac{1}{\varepsilon^{n-1}2\pi^{\frac{n}{2}}}\int_{|x-z|=\varepsilon}(f(z)-f(x))\int_0^\infty\phi(\frac{\varepsilon^2}{4u})e^{-u}u^{\frac{n}{2}-1}dud\sigma(z)\\
&&+f(x)\alpha(\varepsilon), \;\;x\in \mathbb R^n,
\end{eqnarray*}
where
$\displaystyle\alpha(\varepsilon)=\frac{1}{\Gamma(\frac{n}{2})}\int_0^\infty\phi(\frac{\varepsilon^2}{4u})e^{-u}u^{\frac{n}{2}-1}du$,
$0<\varepsilon<1$.

Since $f$ is a continuous function we get
$$
\lim_{\varepsilon \rightarrow 0^+}\frac{1}{\varepsilon^{n-1}2\pi^{\frac{n}{2}}}\int_{|x-z|=\varepsilon}(f(z)-f(x))\int_0^\infty\phi(\frac{\varepsilon^2}{4u})e^{-u}u^{\frac{n}{2}-1}dud\sigma(z)=0.
$$
It is clear that $\alpha$ is a bounded function on $(0,\infty)$. Moreover, if there exists $\phi(0^+)=\lim_{t \rightarrow 0^+}\phi(t)$, by using the dominated convergence theorem we obtain
$$
\lim_{\varepsilon \rightarrow 0^+}\alpha(\varepsilon)=\phi(0^+).
$$
Since $\Delta_zW_t(x,z)=\frac{\partial}{\partial t}W_t(x,z)$, $x,z \in \mathbb R^n$ and $t>0$, the above arguments allow us to establish (\ref{des2.6}) and (\ref{des2.7}).

\end{proof}

\def\cprime{$'$}


\begin{thebibliography}{10}

\bibitem{AST}
I.~{Abu-Falahah}, P.~R. {Stinga}, and J.~L. {Torrea}.
\newblock {Square functions associated to Schrodinger operators}.
\newblock {\em Studia Mathematica},  203:171--194, 2011.


\bibitem{BFHR}
J.~J. {Betancor}, J.~C. {Fari{\~n}a}, E.~{Harboure}, and
  L.~{Rodr{\'{\i}}guez-Mesa}.
\newblock {$L^{p}$-boundeness properties of variation operators in the
  Schr\"odinger setting}.
\newblock {\em ArXiv e-prints,  arXiv:1010.3117}, Oct. 2010.

\bibitem{Bo1}
J.~Bourgain.
\newblock Some remarks on {B}anach spaces in which martingale difference
  sequences are unconditional.
\newblock {\em Ark. Mat.}, 21(2):163--168, 1983.

\bibitem{Bu1}
D.~L. Burkholder.
\newblock A geometrical characterization of {B}anach spaces in which martingale
  difference sequences are unconditional.
\newblock {\em Ann. Probab.}, 9(6):997--1011, 1981.

\bibitem{Bu3}
D.~L. Burkholder.
\newblock Martingales and {F}ourier analysis in {B}anach spaces.
\newblock In {\em Probability and analysis ({V}arenna, 1985)}, volume 1206 of
  {\em Lecture Notes in Math.}, pages 61--108. Springer, Berlin, 1986.

\bibitem{CW}
R.~R. Coifman and G. Weiss.
\newblock {\em Transference metods in analysis.}
\newblock C.M.B.S. regional conference series, 31, Amer. Math. Soc., Providence, R.I., 1971.

\bibitem{DGMTZ}
J.~Dziuba{\'n}ski, G.~Garrig{\'o}s, T.~Mart{\'{\i}}nez, J.~L. Torrea, and
  J.~Zienkiewicz.
\newblock {$BMO$} spaces related to {S}chr\"odinger operators with potentials
  satisfying a reverse {H}\"older inequality.
\newblock {\em Math. Z.}, 249(2):329--356, 2005.

\bibitem{GLY}
L. Grafakos, L. Liu, and D. Yang.
\newblock  Vector valued singular integrals and maximal functions on spaces of homogeneous type.
\newblock {\em Math. Scand.} 104 (2):296--310, 2009.

\bibitem{GD}
S.~Guerre-Delabri{\`e}re.
\newblock Some remarks on complex powers of {$(-\Delta)$} and {UMD} spaces.
\newblock {\em Illinois J. Math.}, 35(3):401--407, 1991.

\bibitem{Hy1}
T.~P. Hyt{\"o}nen.
\newblock Aspects of probabilistic {L}ittlewood-{P}aley theory in {B}anach
  spaces.
\newblock In {\em Banach spaces and their applications in analysis}, pages
  343--355. Walter de Gruyter, Berlin, 2007.

\bibitem{Hy2}
T.~P. Hyt{\"o}nen.
\newblock Littlewood-{P}aley-{S}tein theory for semigroups in {UMD} spaces.
\newblock {\em Rev. Mat. Iberoam.}, 23(3):973--1009, 2007.

\bibitem{MTX}
T.~Mart{\'{\i}}nez, J.~L. Torrea, and Q.~Xu.
\newblock Vector-valued {L}ittlewood-{P}aley-{S}tein theory for semigroups.
\newblock {\em Adv. Math.}, 203(2):430--475, 2006.

\bibitem{RbRT}
J.~L. Rubio de Francia, F. Ruiz, and J.L. Torrea.
\newblock { Calderón-Zygmund theory for operator-valued kernels.  }
\newblock {\em Adv. in Math.}, 62 (1):7--48, 1986.

\bibitem{Sh1}
Z.~W. Shen.
\newblock {$L^p$} estimates for {S}chr\"odinger operators with certain
  potentials.
\newblock {\em Ann. Inst. Fourier (Grenoble)}, 45(2):513--546, 1995.

\bibitem{SWe} E.~M. Stein and G. Weiss.
\newblock {\em Introduction to Fourier analysis on Euclidean spaces.}
\newblock Princeton Mathematical Series, 32, Princeton University Press,
  Princeton, N.J., 1971.

\bibitem{Stein}
E.~M. Stein.
\newblock {\em Topics in harmonic analysis related to the {L}ittlewood-{P}aley
  theory.}
\newblock Annals of Mathematics Studies, No. 63. Princeton University Press,
  Princeton, N.J., 1970.

\bibitem{StTo}
K.~Stempak and J.~L. Torrea.
\newblock Poisson integrals and {R}iesz transforms for {H}ermite function
  expansions with weights.
\newblock {\em J. Funct. Anal.}, 202(2):443--472, 2003.

\bibitem{Xu}
Q.~Xu.
\newblock Littlewood-{P}aley theory for functions with values in uniformly
  convex spaces.
\newblock {\em J. Reine Angew. Math.}, 504:195--226, 1998.

\bibitem{Zim} F. Zimmermann.
\newblock On vector valued Fourier multiplier theorems.
\newblock {\em Studia Math.}, 93:201--222, 1989.

\end{thebibliography}
\end{document}